\documentclass[12pt,reqno]{amsart}
\usepackage{amsmath}
\usepackage{amssymb}
\usepackage{amsrefs}
\usepackage{stmaryrd}
\usepackage[all]{xy}
\usepackage{a4wide}
\usepackage[utf8]{inputenc}

\makeatletter
\@namedef{subjclassname@2010}{  \textup{2010} Mathematics Subject Classification}
\makeatother
\newtheorem{theorem}{Theorem}[section]
\newtheorem{lem}[theorem]{Lemma}

\newtheorem{prop}[theorem]{Proposition}
\newtheorem{cor}[theorem]{Corollary}
\theoremstyle{definition}
\newtheorem{defn}[theorem]{Definition}
\newtheorem{ex}[theorem]{Example}

\newtheorem{example}[theorem]{Example}

\numberwithin{equation}{section}

\def\z*{\mbox{Z$^*$}}
\def\s*{\mbox{(S$^*$)}}

\renewenvironment{proof}{\par\noindent{\bf Proof \,}}{$\hfill \Box$\par\bigskip}

\begin{document}
\title{The dual notion of strong irreducibility}
\author{Jawad Abuhlail}
\address{Department of Mathematics and Statistics\\
Box $\#$ 5046, KFUPM, Dhahran, KSA}
\email{abuhlail@kfupm.edu.sa}
\urladdr{http://faculty.kfupm.edu.sa/math/abuhlail/}
\author{Christian Lomp}
\curraddr{Center of Mathematics\\
University of Porto, Portugal}
\email{clomp@fc.up.pt}
\urladdr{http://www.fc.up.pt/pessoas/clomp}
\date{\today }
\subjclass[2010]{Primary 16D10; Secondary 13C05, 13C13, 54B99}
\keywords{Strongly irreducible ideals, strongly irreducible submodules,
strongly hollow submodule, Zariski Topology, Dual Zariski Topology}
\thanks{The first author would like to acknowledge the support provided by
the Deanship of Scientific Research (DSR) at King Fahd University of
Petroleum $\&$ Minerals (KFUPM) for funding this work through project
SB101023. Compiling this manuscript started during the visit of the second
author to KFUPM. He would like to thank KFUPM for their hospitality. The
second author was partially supported by Centro de Matem\'{a}tica da
Universidade do Porto (CMUP), financed by FCT (Portugal) through the
programs POCTI (Programa Operacional Ci\^{e}ncia, Tecnologia, Inova\c{c}\~{a}%
o) and POSI (Programa Operacional Sociedade da Informa\c{c}\~{a}o), with
national and European community structural funds.}

\begin{abstract}
This note gives a unifying characterization and exposition of strongly
irreducible elements and their duals in lattices. The interest in the study
of strong irreducibility stems from commutative ring theory, while the dual
concept of strong irreducibility had been used to define Zariski-like
topologies on specific lattices of submodules of a given module over an
associative ring. Based on our lattice theoretical approach, we give a
unifying treatment of strong irreducibility, dualize results on strongly
irreducible submodules, examine its behavior under central localization and
apply our theory to the frame of hereditary torsion theories.
\end{abstract}

\maketitle

\section{Irreducibility in semilattices}

\subsection{Introduction}

A \emph{lower semilattice} $(L,\wedge )$ is a partially ordered set $L$ such
that for any two elements $a,b\in L$ there exists a greatest lower bound $%
a\wedge b$. An \emph{upper semilattice} $(L,\vee )$ is defined analogously
asking that any two elements have a least upper bound $a\vee b$. For any $%
a,b\in L$ we set the interval of $a$ and $b$ to be the subset 
\begin{equation*}
\lbrack a,b]=\{x\in L\mid a\leq x\leq b\}.
\end{equation*}

\begin{defn}
Let $\mathcal{L}=(L,\wedge )$ be a lower semilattice. An element $p\in L$ is
called \textit{irreducible} if for any $a,b\in L$ with $p\leq a,b$: 
\begin{equation}
a\wedge b\leq p\qquad \Rightarrow \qquad a\leq p\mbox{ or }b\leq p.
\label{equation_1}
\end{equation}%
The element $p$ is called \textit{strongly irreducible} if Equation (\ref%
{equation_1}) holds for any $a,b\in L$.
\end{defn}

Strongly irreducible ideals and submodules have been studied in %
\cites{Atani, Azizi,HeinzerRatliffRush,KhaksariErshadSharif}. The dual
notion of a strongly irreducible submodule was termed \emph{strongly hollow}
in \cite{Abuhlail_Zariski2} and our purpose is to use lattice theory to
obtain unifying results on strongly irreducible elements either in the
lattice of one-sided or two-sided ideals, submodules or in the dual lattices
of those. We will also apply our results to the lattice of hereditary
torsion theory. Note that strongly irreducible elements are called \textit{%
prime elements} in \cite{Markowsky}. The reader might be warned that the
term \textit{irreducible element} is usually used for a meet- or
join-irreducible element in lattice theory (see \cite[p. 102]{Gratzer}): an
element $p$ of a lattice $\mathcal{L}=(L,\wedge ,\vee ,0,1)$ is called
meet-irreducible if $p\neq 1$ and whenever $p=a\wedge b$ for elements $%
a,b\in L$ then $p=a$ or $p=b$. Albu and Smith call a submodule of a module 
\textit{irreducible} if it is meet-irreducible in the lattice of submodules
(see \cite{AlbuSmith}). A join-irreducible element in $\mathcal{L}$ is a
meet-irreducible element in the dual lattice $\mathcal{L}^{\circ }$.

Prime ideals in a ring $R$ are strongly irreducible elements in the lattice
of ideals of $R$. This property allows that the basis of Zariski-closed
subsets of $\mathrm{Spec}(R)$ satisfy the axioms of a topology.

\begin{example}
\label{zariski1} Let $\mathcal{L}=(L,\wedge ,\vee ,0,1)$ be a complete
lattice and let $X\subseteq L\setminus \{1\}$ be a non-empty set of strongly
irreducible elements. Define for all $a\in L$ 
\begin{equation*}
\mathcal{V}(a)=\{p\in X\mid a\leq p\}.
\end{equation*}%
Then $\{\mathcal{V}(a)\mid a\in L\}$ is a basis of closed sets of a topology
on $X$, because for $a,b\in L$ certainly $\mathcal{V}(a)\cup \mathcal{V}%
(b)\subseteq \mathcal{V}(a\wedge b)$. If $p\in \mathcal{V}(a\wedge b)$, then 
$a\wedge b\leq p$ and as $p$ is strongly irreducible, $a\leq p$ or $b\leq p$%
, \emph{i.e.} $p\in \mathcal{V}(a)\cup \mathcal{V}(b)$. Hence $\mathcal{V}%
(a)\cup \mathcal{V}(b)=\mathcal{V}(a\wedge b).$ It is clear that $%
\bigcap_{a\in A}\mathcal{V}(a)=\mathcal{V}(\bigvee A)$ for any $A\subseteq L$
and that $\mathcal{V}(0)=X$ while $\mathcal{V}(1)=\emptyset $. Of course
this is the prototype of the Zariski topology for $\mathcal{L}$ being the
lattice of ideals of a commutative ring $R$ and $X=\mathrm{Spec}(R)$ being
the set of prime ideals. The same construction was used in \cite[Sec. 4]%
{Azizi} to topologize the space of strongly irreducible ideals of a
commutative ring and in \cite{Simmons} to topologize the space of
irreducible hereditary torsion theories over a ring.
\end{example}

\begin{example}
\label{zariski2} Let $\mathcal{L}=(L,\wedge ,\vee ,0,1)$ be a complete
lattice and let $X\subseteq L\setminus \{0\}$ be a non-empty set of elements
that are strongly irreducible in the dual lattice $\mathcal{L}^{\circ }$.
Hence an element $p\in X$ satisfies for all $a,b\in L$: 
\begin{equation}
p\leq a\vee b\qquad \Rightarrow \qquad p\leq a\mbox{ or }p\leq b
\end{equation}%
Define for all $a\in L$ 
\begin{equation*}
\mathcal{\chi }(a)=\{p\in X\mid p\not\leq a\}.
\end{equation*}%
Then $\{\mathcal{\chi }(a)\mid a\in L\}$ is a basis of open sets of a
topology in $X$ where for $a,b\in L$ and $p\in \mathcal{\chi }(a)\cap 
\mathcal{\chi }(b)$ one has $p\not\leq a\vee b$ as $p$ is strongly
irreducible in $\mathcal{L}^{\circ }$. Thus $\mathcal{\chi }(a)\cap \mathcal{%
\chi }(b)=\mathcal{\chi }(a\vee b)$. It is clear that $\bigcup_{a\in A}%
\mathcal{\chi }(a)=\mathcal{\chi }(\bigwedge A)$ for any $A\subseteq L$ and
that $\mathcal{\chi }(1)=\emptyset $ while $\mathcal{\chi }(0)=X$. This
topology has been used in \cite{Abuhlail_DualZariski} to define a
Zariski-like topology on the spectrum of second submodules of a non-zero
module over an associative ring.
\end{example}

\subsection{Properties of irreducible elements}

In what follows we will show some general properties of strongly irreducible
elements in lattices that will be applied later to the case of lattices of
submodules of a given module over an associative ring. Recall that a \textit{%
waist} (or \textit{node}) in a partially ordered set is an element that is
comparable to any other element.

\begin{lem}
\label{lemma1} Let $L$ be a lower semilattice and $a\leq p \leq b$ elements
in $L$.

\begin{enumerate}
\item If $p$ is (strongly) irreducible in $L$, then $p$ is also (strongly)
irreducible in \newline
$\{ x\in L \mid x\leq b\}$ and in $\{ x\in L \mid a\leq x\}$.

\item If $p$ is strongly irreducible in $L$, then it is also irreducible.

\item If $p$ is irreducible in $L$ and a waist in $L$ then $p$ is strongly
irreducible in $L$.
\end{enumerate}
\end{lem}

\begin{proof}
(1) and (2) are clear; (3) Suppose that $p$ is irreducible and a waist. For
any $a,b\in L$ with $a\not\leq p$ and $b\not\leq p$, we have $p\leq a$ and $%
p\leq b$ since $p$ is a waist. As $p$ is irreducible, $a\wedge b\not\leq p.$
\end{proof}

\subsection{Prime elements}

A partially ordered set $(L,\leq )$ is called a \emph{partially ordered
groupoid} if there exists an associative binary operation $\ast :L\times
L\rightarrow L$ such that for all $a,b,c\in L$: $a\leq b$ implies $a\ast
c\leq b\ast c$ and $c\ast a\leq c\ast b$. An element $p\in L$ is called a 
\emph{prime element} if for all $a,b\in L$: $a\ast b\leq p\Rightarrow a\leq
p $ or $b\leq p$.

\begin{lem}
\label{prime_elements} Let $(L,\leq,\ast)$ be a partially ordered groupoid
such that $(L,\leq)$ is a lower semilattice. If $a \ast b \leq a\wedge b$
for all $a,b\in L$, then any prime element in $(L,\ast)$ is a strongly
irreducible element in $(L,\wedge)$.
\end{lem}

\begin{example}
In \cite{Bican}, Bican et al. equipped the lattice $\mathcal{L}_{2}(M)$ of
fully invariant submodules of a module $M$ over a ring $R$ with the
structure of a partially ordered groupoid that satisfies the condition of
the lemma: For any $N,K\in \mathcal{L}_{2}(M)$ set 
\begin{equation*}
N\star _{M}K:=\sum \left\{ (N)f\mid f\in \mathrm{Hom}_{R}(M,K)\right\}
\subseteq N\cap K.
\end{equation*}%
Hence prime elements of $\mathcal{L}_{2}(M)$ are strongly irreducible and
the Zariski-like topology of such submodules defined as in Example \ref%
{zariski1} has been considered in \cite{Abuhlail_Zariski1}.

A multiplication module over a ring $R$ is a module $M$ such that any
submodule is of the form $IM$ for some ideal $I$ of $R$. In particular, any
submodule of a multiplication module is fully invariant. If $R$ is
commutative, then any multiplication module is a self-generator module, 
\emph{i.e.} $M\star _{M}K=M$ for any submodule $K$ of $M$. It has been shown
in \cite[3.2]{Lomp} that for any multiplication module which is a
self-generator module the $\star _{M}$-product is given as follows: 
\begin{equation*}
(IM)\star _{M}(JM)=(IJ)M.
\end{equation*}%
Thus prime submodules of such $M$ must have the form $PM$ with $P$ an ideal
of $R$ such that $\overline{P}$ is a prime ideal of $\overline{R}$ where $%
\overline{\phantom{x}}:R\rightarrow R/\mathrm{Ann}_{R}(M)$ is the canonical
projection.
\end{example}

\begin{example}
Let $M$ be a module over an associative ring $R.$ A dual operation on $%
\mathcal{L}_{2}(M)$ has been defined by Bican et al. \cite{Bican} equipping
the dual lattice $\mathcal{L}_{2}(M)^{\circ }$ with the structure of a
partially ordered groupoid that satisfies the condition of the lemma: For
any $N,K\in \mathcal{L}_{2}(M)$ set 
\begin{equation*}
N\ {\square _{M}}\ K:=\bigcap \left\{ (N)f^{-1}\mid f\in \mathrm{Hom}%
_{R}(M/K,M)\right\} \supseteq N+K.
\end{equation*}%
Hence prime elements of $\mathcal{L}_{2}(M)^{\circ }$ are strongly
irreducible. The Zariski-like topology considered in \cite%
{Abuhlail_DualZariski} on the set of the so-called \emph{second submodules}
of $M$ coincides with the topology defined in Example \ref{zariski2}.

A comultiplication module over a ring $R$ is a module $M$ such that any
submodule is of the form $\mathrm{Ann}_{M}(I)$ for some ideal $I$ of $R$.
Clearly any submodule of a comultiplication module is fully invariant. It
can be shown that for any comultiplication module which is a
self-cogenerator module, \emph{i.e.} $0\ {\square _{M}}\ K=K$ for any $%
K\subseteq M$, the $\square _{M}$-product is given as follows: 
\begin{equation*}
\mathrm{Ann}_{M}(I)\ {\square _{M}}\ \mathrm{Ann}_{M}(J)=\mathrm{Ann}%
_{M}(JI).
\end{equation*}%
Yassemi's dual prime submodules provide another source of submodules of a
given module $M$ being strongly irreducible in the dual lattice of $\mathcal{%
L}(M)$ (see \cite{Yassemi}).
\end{example}

\subsection{Total orderings}

The lattice of ideals of chain rings and the lattice of subcomodules of
unserial coalgebras are examples of lattices whose ordering is total. This
total ordering is synonymous to the condition that every element is strongly
irreducible.

\begin{lem}
Let $(L,\leq)$ be a lower semilattice and $a,b\in L$. Then $a\wedge b$ is
strongly irreducible in $L$ if and only if $a\leq b$ and $a$ is strongly
irreducible in $L$ or $b\leq a$ and $b$ is strongly irreducible in $L$
\end{lem}

\begin{proof}
If $c=a\wedge b$ is strongly irreducible, then $a\leq c$ or $b\leq c$. This
shows $a\leq b$ or $b\leq a$. Thus $a=c$ or $b=c$. The converse is trivial.
\end{proof}

\begin{cor}
\label{total} Every element of a lower semilattice $(L,\leq )$ is strongly
irreducible if and only if $\leq$ is a total ordering.
\end{cor}

\begin{proof}
Let $a$ and $b$ be elements of $L$ and set $c=a\wedge b$. If $a\not\leq b,$
then $a\not\leq c$. Since $a\wedge b=c$ and $c$ strongly irreducible, we
have $b\leq c$, i.e. $b\leq a$. This shows that $\leq $ is a total ordering.
The converse is clear.
\end{proof}

\subsection{Irreducible elements in complete lattices}

A lower semilattice $\mathcal{L}=(L,\wedge )$ resp. an upper semilattice $%
\mathcal{L}=(L,\vee )$ is \textit{complete} if arbitrary meets $\bigwedge A$
resp. joins $\bigvee A$ exist for subsets $A\subseteq L$. One defines then $%
0=\bigwedge L$ resp. $1=\bigvee L$. Note that a lower semilattice $(L,\wedge
)$ with a smallest element $0$ is called \textit{uniform} if $0$ is
irreducible in $L$. Any complete lower semilattice $\mathcal{L}=(L,\wedge )$
can be made into a lattice by setting $a\vee b=\bigwedge \{c\in L\mid a\leq c%
\mbox{and }b\leq c\}$ for all $a,b\in L$.

One says that an element $p$ of a complete lattice $\mathcal{L}=(L,\wedge
,\vee ,0,1)$ is weakly $\wedge $-distributive if whenever $x\wedge y=0$, one
has $p=(x\vee p)\wedge (y\vee p)$. Similarly an element $p$ is weakly $\vee $%
-distributive if whenever $x\vee y=1$, then $p=(x\wedge p)\vee (y\wedge p)$.

The first lemma shows that strongly irreducible elements are always weakly $%
\wedge $-distributive.

\begin{lem}
\label{weakdistributive} Let $\mathcal{L}=(L,\wedge,\vee,0,1)$ be a complete
lattice. If $p$ is strongly irreducible in $L$, then it is weakly $\wedge$%
-distributive in $L$.
\end{lem}

\begin{proof}
Let $x\wedge y=0\leq p.$ Then $x\leq p$ or $y\leq p,$ as $p$ is strongly
irreducible in $L,$ and so $x\vee p=p$ or $y\vee p=p$. Hence, $p=(x\vee
p)\wedge (y\vee p)$ which shows that $p$ is weakly distributive.
\end{proof}

In algebraic lattices the irreducibility of an element can be checked on the
set of compact elements. To prepare this result we have the following lemma:

\begin{lem}
\label{lemma1_alg} Let $\mathcal{L}=(L,\wedge ,\vee ,0,1)$ be a complete
lattice and $\mathcal{C}\subseteq L$. Assume that any element is equal to a
join of elements in $\mathcal{C}$. Then an element $p\in L$ is strongly
irreducible in $L$ if and only if Equation (\ref{equation_1}) holds for all
for all elements $a,b\in \mathcal{C}$.
\end{lem}

\begin{proof}
Suppose that Equation (\ref{equation_1}) holds for all elements in $\mathcal{%
C}$. Let $a\wedge b\leq p$. By hypothesis $a=\bigvee C$ and $b=\bigvee D$
with $C,D\subseteq \mathcal{C}$. If $a\not\leq p$, then there exists $c\in C$
with $c\not\leq p$. For any element $d\in D$ we have $c\wedge d\leq a\wedge
b\leq p$ hence, by hypothesis, $d\leq p$ since $c\not\leq p$. Thus $%
b=\bigvee D\leq p$.
\end{proof}

\begin{example}
The lemma above applies in particular to complete algebraic lattices. Recall
that in a complete upper semilattice $(L,\vee ,1)$, an element $c\in L$ is
called \textit{compact} if whenever $c\leq \bigvee A$ for a subset $A$ of $L$%
, there exists a finite subset $A^{\prime }\subseteq A$ such that $c\leq
\bigvee A^{\prime }$. Furthermore, $L$ is called \textit{algebraic} if every
element of $L$ is the join of a set of compact elements (see \cite[I.3.16]%
{Gratzer}).
\end{example}

\begin{prop}
\label{chain} Let $\mathcal{L}=(L,\wedge ,\vee ,0,1)$ be a complete lattice
and $\mathcal{C}\subseteq \mathcal{L}$ a chain of strongly irreducible
elements in $L$. Then $p=\bigwedge C$ is a strongly irreducible element in $L
$.
\end{prop}

\begin{proof}
Let $a,b\in L$ with $a\wedge b\leq p=\bigwedge C$. Then also $a\wedge b\leq q
$ for any $q\in C$. Suppose $a\not\leq p$ so that there exists $q\in C$ with 
$a\not\leq q$. In particular, for all $q^{\prime }\in C$ with $q^{\prime
}\leq q$ we also have $a\not\leq q^{\prime }$. Set $C^{\prime }:=C\cap
\lbrack 0,q]$. Since $C^{\prime }$ consists of strongly irreducible
elements, we have $b\leq q^{\prime }$ for all $q^{\prime }\in C^{\prime }$.
Since $C$ is a chain, $b\leq \bigwedge C^{\prime }=\bigwedge C=p$.
\end{proof}

We recover the fact noted in \cite[Theorem 2.1]{Azizi} that over every
proper ideal in a commutative ring lies a minimal strongly irreducible ideal.

\begin{cor}
\label{minimal} Let $\mathcal{L}=(L,\wedge ,\vee ,0,1)$ be a complete
lattice and $a\in L$. If $a$ is bounded from above by a strongly irreducible
element, then there exists a minimal strongly irreducible element $p$ in $L$
with $a\leq p$.
\end{cor}

\begin{proof}
The hypothesis implies that $T=\{ p\in [a,1] \mid p 
\mbox{ is strongly
irreducible in } L \}$ is non-empty. Equipping $T$ with the opposite partial
ordering, Proposition \ref{chain} allows us to apply Zorn's Lemma to obtain
a minimal element in $T$.
\end{proof}

Given a non-empty subset $A$ of a complete lattice $(L,\wedge ,\vee ,0,1)$,
we set 
\begin{equation*}
\Omega (A)=\{x\in L\setminus \{0\}\mid a\wedge x=0\mbox{ for all }a\in A\}.
\end{equation*}%
For an element $p\in L$ we simply write $\Omega (p):=\Omega (\{p\})$. An
element $p$ is called \emph{\ essential in $L$} if $\Omega (p)=\emptyset $.
Note that if $L$ is uniform, then any nonzero element of $L$ is essential in 
$L$. Recall that in a \emph{pseudo-complement} of an element $a$ in a
semilattice $L$ is (if it exists) the greatest element $x$ such that $%
a\wedge x=0$. If $L$ is a complete lattice and $x$ is a pseudo-complement of 
$a$ in $L$, then $x=\bigvee \Omega (a)$.

\begin{lem}
\label{characterisation_str_irr} Let $L$ be a complete lattice $\mathcal{L}%
=(L,\wedge,\vee, 0,1 )$. The following statements are equivalent for an
element $p\in L$ with $\Omega(p)\neq \emptyset$:

\begin{enumerate}
\item[(a)] $p$ is strongly irreducible in $L$.

\item[(a')] $L\setminus [0,p]$ is closed under $\wedge$.

\item[(b)] $(\Omega(p),\wedge)$ is a semilattice and for any $a\not\in [0,p]$
there exists $y\in \Omega(p)$ with $y\leq a$.

\item[(c)] $p$ is a pseudo-complement of some element $q$ in $L$ such that $%
[0,q]$ is uniform.

\item[(d)] for any $q\in \Omega (p)$: $p$ is a pseudo-complement for $q$ in $%
L$ and $[0,q]$ is uniform.

\item[(e)] $p$ is irreducible and weakly $\wedge$-distributive in $L$.
\end{enumerate}
\end{lem}

\begin{proof}
First note that $[0,p]\cap \Omega(p)=\emptyset$, because if $a\leq p$ then $%
a=a\wedge p$ and $a\in \Omega(p)$ would mean $a\wedge p=0$ and $a\neq 0$.

$(a)\Leftrightarrow (a^{\prime })$ If $p$ is strongly irreducible, then for
all $a,b\in L^{\prime }=L\setminus \lbrack 0,p]$ also $a\wedge b\not\leq p$, 
\emph{i.e.} $a\wedge b\in L^{\prime }$. On the other hand, since Equation (%
\ref{equation_1}) just needs to be checked for elements in $L^{\prime }$,
the strongly irreducibility of $p$ follows from $L^{\prime }$ being closed
under $\wedge $.

$(a)\Rightarrow (b)$ For any $a,b\in \Omega (p)$ we have $a,b\not\leq p$ and
therefore $a\wedge b\not\leq p$ and $a\wedge b\in \Omega (p)$, \emph{i.e.} $%
\Omega (p)$ is a subsemilattice of $L$. Fix any $x\in \Omega (p)$. For any $%
a\not\leq p$ we have also $a\wedge x\not\leq p$ since $x\not\leq p$ and $p$
strongly irreducible. Hence $0\neq a\wedge x=y\in \Omega (p)$.

$(b)\Rightarrow (a)$ for any $a,b\in L$ with $a,b\not\leq p$, there exist $%
x,y\in \Omega (p)$ with $x\leq a,y\leq b$. Since $\Omega (p)$ is closed
under $\wedge $, $0\neq x\wedge y\not\leq p$, \emph{i.e.} $a\wedge b\not\leq
p$.

$(a)\Rightarrow (d)$ Take any element $q\in \Omega (p)$. Since $%
[0,q]\setminus \{0\}\subseteq \Omega (p)$ is closed under $\wedge $, $[0,q]$
is uniform. Moreover, for any element $b\in \Omega (q)$ we have $q\wedge
b=0\leq p$, thus by the strongly irreducibility of $p$, $b\leq p$ which
means that $p$ is the greatest element among those in $\Omega (q)$ showing
that $p$ is the pseudo-complement of $q$ in $L$.

$(d)\Rightarrow (c)$ is trivial.

$(c)\Rightarrow (a)$ Suppose that $p$ is a pseudo-complement of $q$ in $L$
with $[0,q]$ being uniform. Then $[0,q]\setminus \{0\}\subseteq \Omega (p)$.
If $a\neq 0$ and $a\wedge q=0$, then $a\in \Omega (q)$ and hence $a\leq p$
as $p$ is the greatest element in $\Omega (q)$. Thus if $a,b\not\leq p$,
then $a\wedge q\neq 0\neq b\wedge q$. Since $[0,q]$ is uniform, one has $%
a\wedge b\wedge q\neq 0$. Since $a\wedge b\wedge q\in \Omega (p)$, we have $%
a\wedge b\wedge q\not\leq p$ and therefore $a\wedge b\not\leq p$, \emph{i.e.}
$p$ is strongly irreducible.

$(a)\Rightarrow (e)$ follows from the lemmas \ref{weakdistributive} and \ref%
{lemma1}.

$(e)\Rightarrow (a)$ Let $p$ be irreducible and weakly $\wedge $%
-distributive in $L$. Since $\Omega (p)\neq \emptyset $ we can choose an
element $q\in \Omega (p)$. Hence $p\wedge q=0$ and $q\neq 0$. Suppose that $%
a\wedge b\leq p$ for some $a,b\in L$. Then $a\wedge b\wedge q=0$. As $p$ is
weakly $\wedge $-distributive, $p=(a\vee p)\wedge ((b\wedge q)\vee p)$.
Suppose that $a\not\leq p$. As $p$ is irreducible and $a\wedge p\not\leq p$, 
$b\wedge q\vee p=p$, i.e $b\wedge q\leq p\wedge q=0$. Again using the weak $%
\wedge $-distributivity of $p$, we have $p=(b\vee p)\wedge (q\vee p)$. Since 
$p$ is irreducible and $q\not\leq p$, $b\vee p=p$, i.e. $b\leq p$.
\end{proof}

The following result describes strongly irreducible elements in general:

\begin{theorem}
\label{characterisation_str_irr_final} Let $L$ be a complete lattice. If $%
p\in L$ is strongly irreducible in $L,$ then $p$ is irreducible and

\begin{itemize}
\item there exists $p^{\prime }<p$ such that $p$ is a pseudo-complement of
some element $q\in [p^{\prime },-]$ with $[p^{\prime },q]$ being uniform or

\item $p$ is a waist.
\end{itemize}
\end{theorem}

\begin{proof}
Let $p$ be strongly irreducible. If $p$ is not essential in $L^{\prime
}=[p^{\prime },-]$ for some $p^{\prime }<p$, then 
\begin{equation*}
\Omega _{L^{\prime }}(p)=\{x\in L^{\prime }\setminus \{p^{\prime }\}\mid
x\wedge p=p^{\prime }\}\neq \emptyset 
\end{equation*}%
and Lemma \ref{characterisation_str_irr} applies showing that $p$ is a
pseudo-complement of some element $q\in \lbrack p^{\prime },-]$ with $%
[p^{\prime },q]$ being uniform. On the other hand, suppose that $p$ is
essential in all semilattices $[p^{\prime },-]$ for any $p^{\prime }<p$. Let 
$x$ be any element in $L$ with $p\not\leq x$. Then $p^{\prime }=p\wedge x<p$%
. Since $p$ is essential in $[p^{\prime },-],$ we have $x=p^{\prime }$, 
\emph{i.e.} $x<p$. Hence $p$ is a waist in $L$.
\end{proof}

A cocompact element $a$ in a complete lattice $\mathcal{L}=(L,\wedge ,\vee
,0,1)$ is a compact element in the dual upper semilattice $\mathcal{L}%
^{\circ }$. If $0$ is a cocompact element of $L$ and $A$ a non-empty subset
of $L$, then $\Omega (A)$ is a lower semilattice if and only if it contains
a least element, which is then necessarily an atom, \emph{i.e.} an element $%
a\neq 0$ such that $[0,a]=\{0,a\}$.

\begin{cor}
\label{characterisation_str_irr_cor} Let $p$ be an element of a complete
lattice $\mathcal{L}=(L,\wedge ,\vee ,0,1)$ with $0$ being cocompact.
Suppose that $\Omega (p)\neq \emptyset $. Then the following are equivalent:

\begin{enumerate}
\item[(a)] $p$ is strongly irreducible in $L$.

\item[(b)] $p$ is a pseudo-complement of an atom $a$.
\end{enumerate}

In this case any element of $L$ is comparable to $a$ or to $p$.
\end{cor}

\begin{proof}
$(a\Rightarrow (b)$ By Lemma \ref{characterisation_str_irr}, $(a)$ implies
that $\Omega(p)$ is a semilattice. If $a=\bigwedge \Omega(p) = 0$, then as $%
0 $ is cocompact, there exist a finite number of elements $a_1,\ldots,
a_m\in \Omega(p)$ such that $a_1\wedge \cdots\wedge a_m = 0 \not \in
\Omega(p)$ which contradicts the fact that $(\Omega(p),\wedge)$ is a
semilattice. Thus $a\neq 0$.

$(b)\Rightarrow (a)$ follows directly from Lemma \ref%
{characterisation_str_irr}.

Furthermore, for any $b\in L:$ if $a\not\leq b$, then $a\wedge b=0\leq p$.
Thus $b\leq p$.
\end{proof}

\subsection{Complete strong irreducibility}

Completely irreducible ideals in commutative rings have been considered in 
\cite{FuchsHeinzerOlberding}. Here we present a lattice theoretical approach
to this notion.

\begin{defn}
Let $\mathcal{L}=(L,\wedge,\vee,0,1)$ be a complete lattice. An element $p$
is called \textit{completely strongly irreducible} (resp. \textit{completely
irreducible}) if 
\begin{equation*}
\bigwedge A \leq p \qquad \Rightarrow \qquad \exists a\in A: a\leq p.
\end{equation*}
holds for any subset $A \subseteq L$ ( resp. any $A\subseteq L$ with $p\leq
a \: \forall a\in A$):
\end{defn}

Completely strongly irreducible elements are called \textit{completely prime}
in \cite{Markowsky}.

\begin{lem}
\label{completely} Let $\mathcal{L}=(L,\wedge ,\vee ,0,1)$ be a complete
lattice and $p\in L$. Suppose that $p=\bigwedge d_{i}$ with $d_{i}$
cocompact elements, e.g. $\mathcal{L}^{\circ }$ being algebraic.

\begin{enumerate}
\item $p$ is completely strongly irreducible in $L$ if and only if $p$ is
cocompact and strongly irreducible in $L$.

\item If $p\leq d_i$ for all $i\in I$, then $p$ is completely irreducible in 
$L$ if and only if $p$ is cocompact and irreducible in $L$.
\end{enumerate}
\end{lem}

\begin{proof}
(1) Since $p=\bigwedge d_{i}$ is completely strongly irreducible, $d_{i}\leq
p$ for some $i\in I$. Since $d_{i}\leq p\leq \bigwedge d_{i}\leq d_{i}$ we
have equality, \emph{i.e.} $p=d_{i}$ is cocompact. On the other hand,
suppose that $p$ is cocompact and strongly irreducible in $L$. For any
subset $A\subseteq L,$ if $\bigwedge A\leq p$, then there exists a finite
subset $\{a_{1},\ldots ,a_{m}\}\subseteq A$ such that $a_{1}\wedge \cdots
\wedge a_{m}\leq p$. By induction and by strong irreducibility of $p$ there
exists $1\leq i\leq n$ such that $a_{i}\leq p$.

Case (2) is similar.
\end{proof}

\subsection{Distributivity and irreducibility}

A lattice $L$ is called {\emph{distributive}} if $(a\vee p)\wedge (b\vee p)
= (a\wedge b)\vee p$ holds for any $a,b,p\in L$. Rings whose lattice of
ideals is distributive are sometimes called {\emph{arithmetic}}.

\begin{lem}
Let $\mathcal{L}=(L,\wedge ,\vee ,0,1)$ be a complete lattice and $p\in L$.
Suppose that $p$ satisfies $(a\wedge b)\vee p=(a\vee p)\wedge (b\vee p)$ for
any $a,b\in L$ with $a\wedge b\leq p$. Then $p$ is strongly irreducible in $L
$ if and only if it is irreducible in $L$.
\end{lem}

\begin{proof}
Suppose that $p$ is irreducible in $L$. For any $a,b\in L$ with $a\wedge
b\leq p$ we have by hypothesis 
\begin{equation*}
(a\vee p)\wedge (b\vee p)=(a\wedge b)\vee p\leq p.
\end{equation*}%
By irreducibility $a\vee p\leq p$ or $b\vee p\leq p$, \emph{i.e.} $a\leq p$
or $b\leq p$.
\end{proof}

The following corollary yields \cite[Lemma 2.2.]{HeinzerRatliffRush} and 
\cite[Prop. 3.4]{Atani}.

\begin{cor}
\label{distributive} Any (completely) irreducible element in a distributive
lattice is (completely) strongly irreducible.
\end{cor}

\subsection{Strong Kuros-Ore dimension}

Let $(L,\wedge ,\vee ,0,1)$ be a complete \emph{modular} lattice (\emph{i.e. 
}$(x\wedge y)\vee (x\wedge z)=x\wedge (y\vee (x\wedge z))$ for all $x,y,z\in 
\mathcal{L}$). An \emph{irredundant }$\wedge $\emph{-representation} of an
element $x\in L$ is by definition any non-empty set of elements $P\subseteq L
$ with $x=\bigwedge P$ and $x\neq \bigwedge P^{\prime }$ for any proper
subset $P^{\prime }$ of $P$. If $P$ is a finite set, we say that $x$ has a
finite irredundant $\wedge $-representation. According to \cite%
{GrzeszczukOkninskiPuczylowski}, the Kuros-Ore Theorem for a complete \emph{%
modular} lattice $L$ says that any two finite irredundant $\wedge $%
-representations of $0$ consisting of irreducible elements have the same
number of terms. Moreover one says that $L$ has \emph{Kuros-Ore dimension} $n
$ if there exists an irredundant $\wedge $-representation of $0$ of $n$
irreducible elements. It has been shown in \cite[3.2]%
{GrzeszczukOkninskiPuczylowski} that if $L$ has Kuros-Ore dimension $n$,
then the Goldie dimension of $L$ is $n$. 

The question \cite[Question 2.12]{HeinzerOlberding} asks to characterize the
commutative rings such that every ideal can be represented uniquely as an
irredundant intersection of irreducible ideals. Here we will briefly discuss
a version of the Kuros-Ore Theorem for strongly irreducible elements:

\begin{prop}
\label{strongKurosOre} Let $L$ be a complete modular lattice such that 
\begin{equation*}
p_{1}\wedge \cdots \wedge p_{n}=0=q_{1}\wedge \cdots \wedge q_{m}.
\end{equation*}%
are two irredundant $\wedge $-representations of strongly irreducible
elements. Then $n=m$ and there exists a permutation $\pi \in S_{n}$ with $%
q_{i}=p_{\pi (i)}$.
\end{prop}

\begin{proof}
Since strongly irreducible elements are irreducible, $n=m$ by the Kuros-Ore
Theorem. Let $1\leq i\leq n$. Since $q_{i}$ is strongly irreducible and $%
p_{1}\wedge \cdots \wedge p_{n}=0\leq q_{i}$, there exists an index $j=\pi
(i)$ such that $p_{\pi (i)}\leq q_{i}$. Equally, there exists $1\leq \tau
(j)\leq n$ such that $q_{\tau (\pi (i))}=q_{\tau (j)}\leq p_{j}\leq q_{i}$.
If $\tau (j)=\tau (\pi (i))\neq i$, then $q_{i}$ could be dropped from the $%
\wedge $-representation of $0$, what is impossible. Hence $\tau (\pi (i))=i$
proves that $\pi $ is a permutation.
\end{proof}

\section{The lattice of submodules of a module}

The aim of this section is to apply the lattice theoretical notion of
strongly irreducibility and its dual to the lattice of submodules of a
module over an associative ring. As mentioned before, strongly irreducible
submodules had been considered by several authors in 
\cites{ Atani,
Azizi,HeinzerRatliffRush,KhaksariErshadSharif} which our results in the
first section extend from the lattice of submodules to general lattices.

\subsection{Strongly irreducible submodules}

Let $R$ be an associative ring with unity and $M$ a left $R$-module. The set
of submodules $\mathcal{L}(M)$ forms a complete (modular) lattice with $\cap$
as meet $\wedge$ and $+$ as join $\vee$. Since the compact elements of $%
\mathcal{L}(M)$ are the finitely generated submodules and since any
submodule is the sum of cyclic ones, $\mathcal{L}(M)$ is algebraic. The
cocompact elements in $\mathcal{L}(M)$ are those submodules $N$ of $M$ with $%
M/N$ finitely cogenerated and since it is well-known that any submodule $N$
is the intersection of submodules $L_i$ containing $N$ with $M/L_i$ finitely
cogenerated, $\mathcal{L}(M)^\circ$ is also algebraic (see \cite{Wisbauer}*{%
14.9}).

A submodule $N$ of a module $M$ is (completely, strongly) irreducible in $M$
if it is a (completely, strongly) irreducible element in $\mathcal{L}(M)$.

\begin{itemize}
\item Note that $N$ is irreducible in $M$ if and only if $0$ is an
irreducible element in $\mathcal{L}(M/N)$ if and only if $M/N$ is uniform.

\item Finitely cogenerated uniform modules are precisely the subdirectly
irreducible modules, \emph{i.e.} those modules that contain an essential
simple submodule\footnote{%
those modules appear under various names in the literature like cocyclic,
monolithic or colocal}. Hence by Lemma \ref{completely} $N$ is completely
irreducible in $M$ if and only if $M/N$ is subdirectly irreducible.

\item Since $\mathcal{L}(M)$ is algebraic, it follows by Lemma \ref%
{lemma1_alg} that $N$ is strongly irreducible in $M$ if and only if whenever 
$Ra\cap Rb \subseteq N$ for $a,b\in M$, we have $a\in N$ or $b\in N$. This
shows that Lemma \ref{lemma1_alg} extends \cite{Atani}*{2.4}.

\item By Lemma \ref{lemma1}, if $N$ is a waist in $M$ then $N$ is strongly
irreducible in $M$ if and only if $M/N$ is uniform.

\item By Corollary \ref{distributive} the strongly irreducible submodules $N$
of a distributive module $M$ are precisely those with $M/N$ being uniform.

\item Let $N$ be a submodule of a finitely cogenerated module $M$, such that 
$N$ is not essential. Then $0$ is a cocompact element in $\mathcal{L}(M)$
and $\Omega (N)\neq \emptyset $. Applying Corollary \ref%
{characterisation_str_irr_cor}, $N$ is strongly irreducible in $M$ if and
only if $N$ is the unique complement of a simple submodule $A$ of $M$. This
means that $N\oplus A$ is essential in $M$ and if $L$ is a submodule not
containing $A$, then it must be contained in $N$, \emph{i.e.} $N=\sum
\{L\subseteq M\mid A\not\subseteq L\}$.

\item If $N$ is a proper submodule of $M$ that is contained in a strongly
irreducible submodule of $M$, then there exists a minimal strongly
irreducible submodule of $M$ lying over $N$.
\end{itemize}

Note that any proper ideal $I$ of a commutative ring $R$ is contained in a
maximal ideal $P$. Since maximal ideals are prime ideals, it is also
strongly irreducible by Lemma \ref{prime_elements}. Hence property (6)
yields a minimal strongly irreducible ideal over $I$ as observed in \cite%
{Azizi}*{Theorem 2.1}. In the general case of a module over a
non-commutative ring however maximal submodules (\emph{e.g.} maximal left
ideals) might not be strongly irreducible as we will see in Example \ref%
{example_maximal}.

\subsection{Strongly hollow submodules}

We will now apply our irreducible concept to the dual lattice of a module:

\begin{defn}
Let $M$ be a left $R$-module with submodule $N$. If $\mathcal{L}(M)^{\circ }$
is uniform, one calls $M$ \emph{hollow}, while a hollow cyclic module is
called \textit{local}. A submodule $N$ of $M$ that is (completely)
irreducible in $\mathcal{L}(M)^{\circ }$ is called a \textit{(completely)
hollow submodule}, while $N$ is said to be \textit{(completely) strongly
hollow in $M$} if $N$ is (completely) strongly irreducible in $\mathcal{L}%
(M)^{\circ }$.
\end{defn}

This means that $M$ is hollow if and only if $K+L=M\Rightarrow K=M%
\mbox{
or }L=M$ for any submodules $K,L$ of $M$. Also $N$ is strongly hollow in $M$
if and only if for any submodules $K$,$L$ of $M$ 
\begin{equation*}
K+L\subseteq N\qquad \Longrightarrow \qquad K\subseteq N\mbox{ or }%
L\subseteq N.
\end{equation*}

\begin{ex}[see \protect\cite{Abuhlail_Zariski2}]
$L:=\{(x,y)\mid y=x\}\subset \mathbb{R}^{2}$ is a hollow subspace which is
not strongly hollow. The Pr\"{u}fer group $\mathbb{Z}_{p^{\infty }}$ is
strongly hollow (as a submodule of itself) but not completely hollow.
\end{ex}

\begin{lem}
Let $N$ be a submodule of a left $R$-module $M$.

\begin{enumerate}
\item $N$ is a completely hollow submodule if and only if $N$ is a local
submodule.

\item $N$ is a completely strongly hollow submodule if and only if $N$ is
local and a strongly hollow submodule.
\end{enumerate}
\end{lem}

\begin{proof}
Let $\mathcal{L}=\mathcal{L}(M)$. We apply Lemma \ref{completely} to the
dual $\mathcal{L}^\circ$ of the upper semilattice $(\mathcal{L}(M),+,M)$.

(1) By Lemma \ref{completely}(2) $N$ is completely irreducible in $\mathcal{L%
}^\circ$ if and only if $N$ is cocompact in $\mathcal{L}^\circ$ and
irreducible in $\mathcal{L}^\circ$. $N$ being cocompact in $\mathcal{L}%
^\circ $ is equivalent to $N$ being compact in $\mathcal{L}$ which in turn
is equivalent to $N$ being finitely generated. $N$ being irreducible in $%
\mathcal{L}^\circ$ is equivalent to $N$ being hollow. Any hollow module is
local if and only if it is finitely generated.

(2) By the same argument as in (1) using Lemma \ref{completely}(1) instead.
\end{proof}

Applying the dual version of Lemmas \ref{lemma1}, \ref{lemma1_alg} and \ref%
{minimal} to $\mathcal{L}(M)$ we obtain the following lemma:

\begin{lem}
Let $M$ be a left $R$-module with a non-zero submodule $N.$

\begin{enumerate}
\item If $N$ is a strongly hollow submodule, then it is also a hollow module.

\item If $N$ is strongly hollow in $M$, the $N$ is also strongly hollow in $%
L $ and $N/K$ is strongly hollow in $M/K$ for any $K\subseteq N\subseteq L$.

\item $N$ is strongly hollow in $M$ if and only if for all $K,L \subseteq M$
with $M/K, M/L$ being finitely cogenerated: 
\begin{equation*}
N\subseteq K+L \qquad \Longrightarrow \qquad N\leq K \mbox{ or } N\leq L
\end{equation*}

\item If $N$ is a waist, then $N$ is strongly hollow in $M$ if and only if $%
N $ is a hollow module.

\item If $N$ satisfies $(N\cap L)+(N\cap K) = N\cap (L+K)$ whenever $%
N\subseteq L+K$, then $N$ is strongly hollow in $M$ if and only if $N$ is a
hollow module.

\item the strongly hollow submodules of a distributive module are precisely
the hollow submodules.

\item If $N$ contains a strongly hollow submodule, then it contains a
maximal strongly hollow submodule $P$ in $M$.
\end{enumerate}
\end{lem}

\qquad Recall that a module whose lattice of submodules forms a chain is
called \emph{uniserial}. 

\begin{example}
Note that any submodule of a module $M$ is strongly hollow in $M$ if and
only if every submodule of $M$ is irreducible in $M$ if and only if $M$ is
uniserial. Moreover a submodule of a distributive module $M$ is strongly
hollow in $M$ if and only if it is hollow.
\end{example}

\begin{example}
A coalgebra $C$ over a field $K$ is called \emph{\ distributive} if its
lattice of left subcomodules is distributive. The left subcomodules of $C$
can be identified with the right $C^{\ast }=\mathrm{Hom}_{K}(C,K)$-modules
where $C^{\ast }$ becomes an algebra via the convolution product induced by
the comultiplication of $C$. Any distributive coalgebra $C$ decomposes as a
coproduct of chain coalgebras $C=\bigoplus_{I}C_{i}$ (see \cite{LompSantana}*%
{4.5}). In particular, any indecomposable subcomodule has to be a
subcomodule of one of the factors $C_{i}$. Thus, the left subcomodules that
are strongly hollow in $C$ are precisely the subcomodules of the coalgebras $%
C_{i}$ and hence are chain subcoalgebras themselves. In some cases the form
of chain coalgebras over a field can be explicitly stated (see \cite%
{LompSantana} for more details).
\end{example}

We are going now to apply Lemma \ref{characterisation_str_irr} and \ref%
{characterisation_str_irr_cor} to the dual lattice $\mathcal{L}^{\circ }(M)$%
. Recall that a \emph{supplement} of a submodule $K$ of $M$ is a submodule $N
$ that is minimal with respect to $N+K=M$. If the set of possible
supplements of $K$ is a singleton $\{N\}$, then $N$ is called the \emph{%
unique supplement} of $K$. This is equivalent to saying that $N$ is a
pseudo-complement of $K$ in $\mathcal{L}^{\circ }(M)$. Set $\Omega ^{\circ
}(N):=\{K\subsetneq M\mid N+K=M\}$. In general, supplements do not need to
be unique. Modules such that all submodules have unique supplements were
studied by Ganesan and Vanaja \cite{GanesanVanaja}. Weakly distributive
modules do have this property (see \cite{BuyukasikDemirci}). A submodule $U$
of a module $M$ is said to be \emph{weakly distributive} if $U=(U\cap
X)+(U\cap Y)$ for any submodules $X,Y$ with $X+Y=M$. Equivalently $U$ is a
weakly $+$-distributive element in the dual lattice $\mathcal{L}^{\circ }$
of the lattice $\mathcal{L}=(\mathcal{L}(M),\cap ,+,0,M)$.

\begin{prop}
\label{prop_unique_supplement} The following statements are equivalent for a
submodule $P$ of a module $M$ such that $P$ is not small in $M$.

\begin{enumerate}
\item[(a)] $P$ is strongly hollow in $M$

\item[(b)] $\Omega ^{\circ }(P)$ is closed under finite sums and any
submodule not containing $P$ is contained in a member of $\Omega ^{\circ }(P)
$;

\item[(c)] $P$ is a unique supplement of some $L\in\Omega^\circ(P)$ in $M$
such that $M/L$ is hollow.

\item[(d)] for any $L\in \Omega^\circ(P)$: $P$ is a unique supplement of
some $L\in\Omega^\circ(P)$ in $M$ and $M/L$ is hollow.

\item[(e)] $P$ is a hollow and weakly distributive submodule of $M$.
\end{enumerate}

If $M$ is finitely generated, then the following property is equivalent to $%
(a-e)$:

\begin{enumerate}
\item[(f)] $P$ is the unique supplement of a maximal submodule of $M$ in $M.$
\end{enumerate}
\end{prop}

\begin{proof}
Apply Lemmas \ref{characterisation_str_irr} and \ref%
{characterisation_str_irr_cor} to the dual lattice $\mathcal{L}^\circ(M)$.
In particular $\Omega^\circ(P)$ is equal to $\Omega(P)$ in $\mathcal{L}%
^\circ(M)$.
\end{proof}

The module $M$ is called \emph{weakly distributive} if every submodule of $M$
is weakly distributive. Clearly, if $P$ is a supplement of $Q$ in a weakly
distributive module $M$ and $L+Q=M$, then $P=(P\cap L)+(P\cap Q)=P\cap L$ as 
$P\cap Q\ll P$. Hence $P\cap L$ and $P$ is the least element in $\Omega
^{\circ }(Q)$, \emph{i.e.} $P$ is the unique supplement of $Q$ in $M$.

\begin{cor}
If any supplement submodule of a module $M$ is unique, then any hollow
submodule that is not small in $M$ is strongly hollow in $M$.
\end{cor}

\begin{proof}
If $P$ is a hollow submodule of $M$ that is not small in $M$, then there
exist a proper submodule $Q$ of $M$ such that $P+Q=M$. Since $P\cap Q$ is a
proper submodule of $P$, $P\cap Q\ll P$, \emph{i.e.} $P$ is a supplement of $%
Q$ in $M$. By hypothesis $P$ is unique. Moreover $M/Q\simeq P/P\cap Q$ is
hollow. By Proposition \ref{prop_unique_supplement}, $P$ is strongly hollow
in $M$.
\end{proof}

\begin{prop}
\label{stronglyhollowsubmodules} If $P$ is a strongly hollow submodule of $M$%
, then $P$ is a waist in $M$ or $P$ is a unique supplement of a submodule $Q$
in some intermediate submodule $P\subseteq M^{\prime }\subseteq M$ such that 
$M^{\prime }/Q$ is hollow.
\end{prop}

\begin{proof}
This follows from the dual statement of Theorem \ref%
{characterisation_str_irr_final}.
\end{proof}

From Corollary \ref{total} we get the following statement.

\begin{cor}
Let $M$ be a non-zero left $R$-module. Then every non-zero submodule of $M$
is strongly irreducible in $M$ if and only if every submodule of $M$ is
strongly hollow in $M$ if and only if $M$ is uniserial.
\end{cor}

\begin{proof}
Apply Corollary \ref{total} to $\mathcal{L}(M)$ and to its dual $\mathcal{L}%
(M)^\circ$.
\end{proof}

The following property should be compared to Stephenson's characterizations
of distributive modules which says that a module $M$ is distributive if and
only if $\mathrm{Hom}(P/(P\cap Q),Q/(P\cap Q))=0$ for any submodules $P,Q$
of $M$ (see \cite{Stephenson}).

\begin{lem}
\label{Lemma_stephenson} Let $P$ be submodules of a module $M$. If

\begin{enumerate}
\item $P$ is strongly irreducible in $M$ or

\item $P$ is strongly hollow in $M$,
\end{enumerate}

then $\mathrm{Hom}(P/(P\cap Q),Q/(P\cap Q))=0$ for any submodule $Q$ of $M$.
\end{lem}

\begin{proof}
Note that if $P\subseteq Q$, then $P/(P\cap Q)=0$ and the conclusion is
trivially fulfilled. Hence we will assume $P\not\subseteq Q$. Let $%
f:P/(P\cap Q)\rightarrow Q/(P\cap Q)$ and denote by $\pi _{Q}:Q\rightarrow
Q/(P\cap Q)$ resp. $\pi _{P}:P\rightarrow P/(P\cap Q)$ the canonical
projections. Consider 
\begin{equation*}
\Lambda =\{(p,q)\in P\times Q\mid f(\pi _{P}(p))=\pi _{Q}(q)\}.
\end{equation*}%
Let $\mu :\Lambda \rightarrow M$ be the map $\mu (p,q)=p+q$ and set $L:=%
\mathrm{Im}(\mu )$. Note that 
\begin{equation*}
L\cap Q\subseteq P\subseteq L+Q,
\end{equation*}%
because if $p+q\in L\cap Q$, then $p\in P\cap Q$. Hence $0=f(\pi
_{P}(p))=\pi _{Q}(q)$ shows that $q\in P\cap Q$, \emph{i.e.} $L\cap
Q\subseteq P\cap Q\subseteq P$. The second equality follows because $\pi _{Q}
$ is surjective and hence for any $p$ there exists $q$ such that $(p,q)\in
\Lambda $. Thus $p=(p+q)-q\in L+Q$, \emph{i.e.} $P\subseteq L+Q$.

If $P$ is strongly hollow in $M$, then $P\subseteq L$. Hence for any $x\in P$
there exist $(p,q)\in \Lambda$ such that $x=p+q$. Thus $q=x-p\in P\cap Q$
and 
\begin{equation*}
0=\pi_Q(q) = f(\pi_P(p)) = f(\pi_P(x)).
\end{equation*}

If $P$ is strongly irreducible in $M$, then $L\subseteq P$. Hence, for any $%
x\in P$ there exists $(x,q)\in \Lambda $ with $x+q\in L\subseteq P$, \emph{%
i.e.} $q\in P\cap Q$ and $f(\pi _{P}(x))=\pi _{Q}(q)=0$.

In both cases, as $\pi _{P}$ is surjective, we conclude that $f=0$.
\end{proof}

\begin{example}
\label{example_maximal} Let $R=M_{2}(K)$ be the ring of $2\times 2$-matrices
over a field $K$. The left ideals $P=Re_{11}$ resp. $Q=Re_{22}$ consisting
of all matrices whose second resp. first column contains only zero entries
are maximal left ideals of $R$. Clearly $R=P\oplus Q$ and $P\simeq Q$. Thus
by Lemma \ref{Lemma_stephenson}, none of the maximal left ideals $P$ and $Q$
can be strongly irreducible or strongly hollow in $R$. This trivial example
illustrates that strongly irreducibility for non-commutative rings behaves
very differently from strongly irreducibility for commutative rings, where
maximal (and prime) ideals are always strongly irreducible.
\end{example}

Proposition \ref{strongKurosOre} and Lemma \ref{Lemma_stephenson} yield now
the following:

\begin{cor}
Let $M$ be a left $R$-module. If 
\begin{equation*}
P_{1}+\cdots +P_{n}=M=Q_{1}+\cdots +Q_{m}
\end{equation*}%
are two irredundant sums of strongly hollow submodules $P_{i}$ and $Q_{j}$
of $M$, then $n=m$ and there exists a permutation $\sigma \in S_{n}$ such
that $Q_{i}=P_{\sigma (i)}$ for all $i$. If $M$ can be written as a finite
sum of strongly hollow submodules, then $M$ has finite dual Goldie
dimension. Moreover, any strongly hollow submodule of $M$ is contained in
precisely one of the submodules $P_{i}$. Moreover the set $\{P_{1},\ldots
,P_{n}\}$ is unrelated in the sense that for all $i,j$: 
\begin{equation*}
\mathrm{Hom}_{R}(P_{i}/(P_{i}\cap P_{j}),P_{j}/(P_{i}\cap P_{j}))=0
\end{equation*}
\end{cor}

Recall that strongly hollow submodules of a module $M$ are supplements in $M$%
. In some cases supplements are direct summands. For a ring $R$, H. Z\"{o}%
schinger proved in \cite[Satz 2.3]{Zoschinger} that any left ideal which is
a supplement is generated by an idempotent if and only if whenever $P$ is a
projective module with $P/\mathrm{Rad}(P)$ being finitely generated, it is
the case that $P$ is finitely generated - a property that had been
considered by D. Lazard in his work \cite{Lazard} and in his honor a ring
satisfying this condition is called an $L$\emph{-ring}. Hence semiperfect
and rings with zero Jacobson radical are $L$-rings. S. J\o ndrup \cite%
{Jondrup} showed that every PI-ring, \emph{e.g.} commutative ring, is an $L$%
-ring. A ring $R$ is called \emph{local} if $R/\mathrm{Jac}(R)$ is division
ring.

\begin{cor}
Let $R$ be an $L$-ring. If $I$ is a strongly hollow left ideal that is not
contained in the Jacobson radical of $R$, then $I$ is generated by an
idempotent $e^{2}=e$ such that $eR(1-e)=0$. In particular, if $R$ is
commutative, then $R\simeq I\times R^{\prime }$ with $I$ being a local ring.
\end{cor}

\begin{example}
The condition $\mathrm{Hom}_{R}(P/(P\cap Q),Q/(P\cap Q))=0$ for all
submodules $Q$ of $M$ is in general not sufficient to guarantee $P$ to be
strongly irreducible resp. strongly hollow in $M$. This condition is
satisfied for any pair of submodules of a distributive module. Any Pr\"{u}%
fer domain is distributive as a module over itself, but not any ideal of a Pr%
\"{u}fer domain is irreducible resp. local. If $K$ is a field and $R=K[x]$,
then an ideal $I$ of $R$ is irreducible if and only if it is generated by an
irreducible polynomial. The only strongly hollow submodule, \emph{i.e.}
local ideal of $R$ is $0$ since any ideal $I=Rf$ with $0\neq f\in R$ can be
written as the sum of two ideals $I=Rxf+R(1-x)f$. To give another elementary
example, let $S$ be a simple left $R$-module over a ring $R$ such that $%
\mathrm{Hom}_{R}(S,R)=0$. Let $M=S\oplus R$ be the direct sum of $S$ and $R$%
. Then $S$ satisfies $\mathrm{Hom}_{R}(S,Q)=0$ for any submodule $Q$ of $M$
with $S\not\subseteq Q$, because $\mathrm{Soc}(Q)=\mathrm{Soc}(R)\cap Q$ and
hence $\mathrm{Hom}_{R}(S,Q)=\mathrm{Hom}_{R}(S,\mathrm{Soc}(R)\cap Q)=0$.
Writing $S=Rx$ for some $0\neq x\in S$ and defining $A=R(x,1)$ and $B=R(0,1)$
we see that $S\subseteq A+B$, but $S\not\subseteq A,B$. Hence, $S$ is not
strongly hollow in $M$.
\end{example}

We will examine the problematic of the last example in the following lemmas
where $M=E\oplus D$ and $E$ is a simple submodule which is strongly hollow
in $M$. Recall that the Wisbauer category $\sigma \lbrack D]$ of a module $D$
is the full subcategory of the category of left $R$-modules whose objects
are submodules of factor modules of direct sums of copies of $D$. It is not
difficult to see that if $E$ is a simple left $R$-module, then $E\in \sigma
\lbrack D]$ if and only if $\mathrm{Hom}_{R}(E,D/A)\neq 0$ for some
submodule $A$ of $D$.

\begin{lem}
\label{str_hollow_simple} Let $E$ be a simple left $R$-module and $D$ any
left $R$-module. Then $E$ is strongly hollow in $E\oplus D$ if and only if $%
E\not\in\sigma[D]$. 
\end{lem}

\begin{proof}
Write $M=E\oplus D$. Let $D/A$ be a non-zero factor of $D$ with $A\subset D$%
. For any non-zero $f:E\rightarrow D/A$ set $L=\{(x,y)\in E\oplus D\mid
f(x)=y+A\}$ which is a submodule of $M$. As any $(x,y)\in M$ can be written
as $(x,z)+(0,y-z)\in L+D$ for some $z\in D$ with $f(x)=z+A$, we see that $%
M=L+D$. However $E\not\subseteq D$ and since $f\neq 0$ there exists $x\in E$
with $f(x)=y+A$ for some $y\in D\setminus A$. Thus $(x,0)\not\in L$, \emph{%
i.e.} $E\not\subseteq L$. This shows that $\mathrm{Hom}_{R}(E,D/A)\neq 0$
for some submodule $A$ of $D$ implies $E$ is strongly hollow in $M$.

On the other hand, suppose that $E\subseteq K+L$ for some submodules $K,L$
of $M$. If $E\not\subseteq K$ then $E\cap K=0$. The projection $\pi
_{E}:M\rightarrow D$ yields that $K$ is isomorphic to a submodule $\pi
_{E}(K)$ of $D$. If $E\not\subseteq L$, then $E\cap L=0$ and we get the
following chain of homomorphisms 
\begin{equation*}
E\hookrightarrow K+L\longrightarrow (K+L)/L\simeq K/(K\cap L)\hookrightarrow
D/\pi _{E}(K\cap L)
\end{equation*}%
which yields a non-zero map from $E$ to a factor of $D$.
\end{proof}

\begin{lem}
\label{stronglyhollow_commutative} Let $R$ be commutative, $E$ a simple $R$%
-module and $D$ a finitely generated $R$-module. The following are
equivalent for $M=E\oplus D$:

\begin{enumerate}
\item[(a)] $E$ is a strongly hollow submodule of $M$;

\item[(b)] Any submodule $N$ of $M$ either contains $E$ or is contained in $%
D $;

\item[(c)] $R=\mathrm{Ann}(E) + \mathrm{Ann}(D)$.

\item[(d)] $\mathrm{Ann}(D)\not\subseteq \mathrm{Ann}(E)$.
\end{enumerate}
\end{lem}

\begin{proof}
Let $P=\mathrm{Ann}(E)$ and $Q=\mathrm{Ann}(D)$. By \cite{Wisbauer}*{15.4}, $%
\sigma[D]=R/Q$-Mod.

$(a)\Rightarrow (d)$ As $E$ is strongly hollow in $M,$ it follows from Lemma %
\ref{str_hollow_simple} that $R/P\simeq E\not\in \sigma \lbrack D]=R/Q$-Mod.
Thus $Q\not\subseteq P$.

$(d)\Rightarrow (c)$ is trivial since $P$ is maximal.

$(c)\Rightarrow (b)$ By hypothesis there exist $p\in P$ such that $1-p\in Q$%
. If $N\subseteq M$ and $E\not\subseteq N$, then $N\cap E=0$. For any $0\neq
n=e+d\in N$, where $e\in E,d\in D$, we have $(1-p)n=(1-p)e\in E\cap N=0$.
Hence, $n=pn=pd\in D$, \emph{i.e.} $N\subseteq D$.

$(b)\Rightarrow (a)$ Suppose that $E\subseteq N+K$. Since $E\not\subseteq D$%
, either $N$ or $K$ is not contained in $D$. Hence either $N$ or $K$
contains $E$. 
\end{proof}

\begin{example}
One instance where Lemma \ref{stronglyhollow_commutative} fails is if $%
E\simeq D/V$ for some maximal submodule $V$ of $D$, since then $\mathrm{Ann}%
(D)\subseteq \mathrm{Ann}(E)$.
\end{example}

We illustrate our results by characterizing strongly hollow subgroups $P$ of
finite Abelian groups $A$.

\begin{example}
Recall from \cite{rangaswamy} that the hollow Abelian groups are precisely
the sugroups of the Pr\~{A}%
$\frac14$%
fer groups $\mathbb{Z}_{p^{\infty }}$. Thus any finite strongly hollow
submodule of an Abelian group is a cyclic $p$-group. First suppose that $A$
is a finite Abelian $p$-group for some prime number $p$. Let $Q$ be any
other subgroup of $A$, then $\mathrm{Hom}_{\mathbb{Z}}(P/(P\cap Q),Q/(Q\cap
P))=0$ by Lemma \ref{Lemma_stephenson}. On the other hand there exists
always a non-zero homomorphism between two non-zero finite Abelian $p$%
-groups. Hence $P\cap Q=P$ or $Q\cap P=Q$, i.e. $P\subseteq Q$ or $%
Q\subseteq P$. This shows that $P$ is a waist in $A$. Moreover, since $P$ is
hollow and finite, it is cyclic and thus uniserial by the Fundamental
Theorem of finitely generated Abelian groups. But then $P$ is uniform and
essential in $A$ showing that $A$ is uniform and hence also uniform. By the
Fundamental Theorem, $A$ is cyclic. In general if $P$ is a non-zero strongly
hollow subgroup of a finite Abelian group $A$, then $P$ is a finite $p$%
-group for some prime number $p$. Hence $P$ is contained in the $p$%
-component $A_{p}$ of $A$, which is cyclic as we just saw.

We will show now that a non-zero subgroup $P$ of a finite abelian group $A$
is strongly hollow in $A$ if and only if $P$ is a $p$-group and the $p$%
-component $A_p$ is cyclic. While we just saw the necessity we will show
that this condition is also sufficient. Suppose that $P$ is a non-zero $p$%
-subgroup of $A$ and that $A_p$ is cyclic (thus uniserial). There is nothing
to show if $A=A_p$. Hence assume $A\neq A_p$. As an Abelian group, $A$
decomposes into the direct sum of its $q$-components $A_q$ for prime numbers 
$q$. Let $B$ be the direct sum of all $q$-components $A_q$ with $q\neq p$.
Since $A=A_p\oplus B$ and $\mathbb{Z} = \mathrm{Ann}_{\mathbb{Z}}(A/A_p) + 
\mathrm{Ann}_{\mathbb{Z}}(A/B)$, we have that any subgroup $X$ of $A$
decomposes as $X=(X\cap A_p) + (X\cap B)$. If $P\subseteq X+Y$ for subgroups 
$X,Y$, then $P\subseteq (X\cap A_p) + (Y\cap A_p)$. Since $A_p$ is
uniserial, $P\subseteq X$ or $P\subseteq Y$. Thus $P$ is strongly hollow in $%
A$.

To give some explicit examples: let $p$ and $q$ be prime numbers and set $A=%
\mathbb{Z}_p^n\times \mathbb{Z}_q^m$. If $p=q$, then $0$ is the only
strongly hollow $\mathbb{Z}$-submodule of $A$. If $p\neq q$, then any
non-zero strongly hollow $\mathbb{Z}$-submodules of $A$ is either of the
form $p^k\mathbb{Z}_{p^n}$ for $0\leq k <n$ or of the form $q^k\mathbb{Z}%
_{q^m}$ for $0\leq k < m$.
\end{example}

%
%

\section{Strongly irreducible elements under localization}

The behavior of strongly irreducible ideals in commutative rings had been
studied in \cites{Atani,Azizi,HeinzerRatliffRush}. In this section we will
prove some of their results by lattice theoretical means.

The next (easy) lemma lies at the heart of the correspondence obtained in %
\cites{Atani,Azizi}.

\begin{lem}
\label{lemma2} Let $(L,\wedge )$ and $(L^{\prime },\wedge ^{\prime })$ be
two lower semilattices and $G:L\rightarrow L^{\prime }$ and $F:L^{\prime
}\rightarrow L$ order-preserving maps. Suppose that $G$ is a homomorphism of
semilattices such that $a\leq FG(a)$ for all $a\in L$. Let $p\in L$ with $%
GFG(p)=G(p)$. If $G(p)$ is strongly irreducible in $L^{\prime }$, then $%
FG(p) $ is strongly irreducible in $L$.
\end{lem}

\begin{proof}
Let $a\wedge b\leq FG(p)$. Then 
\begin{equation*}
G(a)\wedge G(b)=G(a\wedge b)\leq GFG(p)=G(p).
\end{equation*}%
As $G(p)$ is strongly irreducible, $G(a)\leq G(p)$ or $G(b)\leq G(p)$. Thus $%
a\leq FG(a)\leq FG(p)$ or $b\leq FG(b)\leq FG(p)$.
\end{proof}

Let $R$ be any ring, $S$ a multiplicatively closed subset of the center of $%
R $ containing $1$, but not containing $0$. Let $M$ be a left $R$-module and
denote by $M_{S}$ the set of equivalence classes of pairs $\frac{m}{s}%
:=(m,s)\in M\times S$ subject to the equivalence relation: $\frac{m}{s}=%
\frac{n}{t}\Leftrightarrow \exists u\in S:(mt-ns)u=0.$ The canonical map $%
\varphi :M\rightarrow M_{S}$ sends $m\in M$ to $\frac{m}{1}$. For $M=R$, the
localization of $R$ by $S$ becomes a ring and $M_{S}$ a left $R_{S}$-module.

The map $F:\mathcal{L}(M)\rightarrow \mathcal{L}(M_{S})$ sending a submodule 
$N$ to $F(N)=N_{S}=\{\frac{n}{s}\in M\mid n\in N,$ $s\in S\}$ preserves the
partial order of the lattice $\mathcal{L}(M)$ as well as sums (joins) and
intersections (meets). 

Consider the map $G:\mathcal{L}(M_S) \rightarrow \mathcal{L}(M)$ given by $%
G(B)=\varphi^{-1}(B) =: B\cap M$ for any $B\in \mathcal{L}(M_S)$, which is
order-preserving. Note that $G$ is a homomorphism of the lower semilattices $%
(\mathcal{L}(M_S),\cap)$ and $(\mathcal{L}(M),\cap)$. We have $FG=id$ and $%
id\leq GF$; in particular $FGF=F$ and $GFG=G$.

The next result generalizes \cite[2.6, 2.7]{Atani}.

\begin{lem}
Let $R$ be a ring and $S$ be a multiplicatively closed subset of the center
of $R$ containing $1$, but not containing $0$. Let $M$ be a left $R$-module.
A submodule $B$ of $M_S$ is strongly irreducible in $M_S$ if and only if $%
B\cap M$ is strongly irreducible in $M$.
\end{lem}

\begin{proof}
By Lemma \ref{lemma2} we have for $B\in \mathcal{L}(M_S)$ that if $%
G(B)=B\cap M$ is strongly irreducible in $M$, then $FG(B)=B$ is strongly
irreducible in $M_S$. Reversing the roles of $F$ and $G$, we get that if $%
B=FG(B)$ is strongly irreducible in $M_S$, then $G(B)=B\cap M$ is strongly
irreducible in $M$.
\end{proof}

Considering the dual lattices $\mathcal{L}(M)^\circ$ and $\mathcal{L}%
(M_S)^\circ$, $F$ and $G$ are still order-preserving maps of upper
semilattices. Although $F$ preserves sums, and hence meets in $\mathcal{L}%
(M)^\circ$, $G$ might not preserve sums and might not establish a
homomorphism of lower semilattices $\mathcal{L}(M_S)^\circ \rightarrow 
\mathcal{L}(M)^\circ$.

\begin{lem}
Let $S$ be a multiplicatively closed subset of $R\setminus \{0\}$ and $B$ a
submodule of $M_S$.

\begin{enumerate}
\item If $B$ is strongly hollow in $M_S$, then $B\cap M$ is strongly hollow
in $M$.

\item If $G$ preserves sums and $B \cap M$ is strongly hollow in $M$, then $%
B $ is strongly hollow in $M_S$.
\end{enumerate}
\end{lem}

\begin{proof}
Consider $F$ resp. $G$ as a order-preserving map $\mathcal{L}(M)^\circ
\rightarrow \mathcal{L}(M_S)^\circ$ resp. $\mathcal{L}(M_S)^\circ
\rightarrow \mathcal{L}(M)^\circ$. (1) $F$ is a homomorphism of the lower
semilattices $(\mathcal{L}(M)^\circ, +)$ and $(\mathcal{L}(M_S)^\circ, +)$;
hence Lemma \ref{lemma2} applies (interchanging the roles of $F$ and $G$).
(2) If $G$ preserves sums, it is a homomorphism of lower semilattices $(%
\mathcal{L}(M)^\circ, +)$ and $(\mathcal{L}(M_S)^\circ, +)$; hence Lemma \ref%
{lemma2} applies.
\end{proof}

\section{The lattice of hereditary torsion theories}

The set $R-\mathrm{Tors}$ of hereditary torsion theories in the category of
left $R$-modules, is a distributive lattice. Hence strongly irreducible and
irreducible elements coincide. Irreducible elements in this lattice have
been characterized in \cite{Golan}*{Chapter 32}. The aim of this section is
to characterize the dual notion of strongly irreducible in this lattice. We
say that a torsion theory $\tau \in R-\mathrm{Tors}$ is {\emph{(completely)
hollow}} in $R-\mathrm{Tors}$ if $\tau $ is (completely) irreducible in the
dual lattice $R-\mathrm{Tors}^{\circ }$.

For all unexplained notions of torsion theory we refer the reader to \cite%
{Golan}. We denote an element $\tau \in R-\mathrm{Tors}$ by some Greek
letter representing either the injective module $E_{\tau }$ that cogenerates
its torsion free class, or its Gabriel filter $\mathcal{G}_{\tau }$ or its
torsion pair $(\mathbb{T}_{\tau },\mathbb{F}_{\tau })$. The partial order in 
$R-\mathrm{Tors}$ is defined by $\tau \leq \sigma $ if and only if $\mathbb{T%
}_{\tau }\subseteq \mathbb{T}_{\sigma }$, for any $\tau ,\sigma \in R$-$%
\mathrm{Tors}$. Alternatively one could have said that $E_{\sigma }\leq
E_{\tau }$. For the zero module $E=0$, one has $\mathrm{Cog}(0)=\{0\}$. Thus
the torsion class associated to $0$ is the whole category $R$-Mod and hence
is the largest torsion theory in $R-\mathrm{Tors}$ denoted by $\mathbf{1}$.
For any injective cogenerator $Q$ of $R$-Mod one has $\mathrm{Cog}(Q)=R$%
-Mod. Thus the torsion class associated to an injective cogenerator contains
just the zero module and hence is the least torsion theory in $R-\mathrm{Tors%
}$ denoted by $\mathbf{0}$. If $U$ is a set of torsion theories with torsion
pairs $(\mathbb{T}_{\tau },\mathbb{F}_{\tau })$ for all $\tau \in U$, then $%
\bigcap_{\tau \in U}\mathbb{T}_{\tau }$ is again a torsion class and $%
\bigcap_{\tau \in U}\mathbb{F}_{\tau }$ is again a torsionfree class (see 
\cite{Golan}*{2.5,2.6}).

\begin{defn}
For any subset $U$ of $R-\mathrm{Tors}$ one defines

\begin{enumerate}
\item $\bigwedge U$ is the torsion theory associated to the torsion class $%
\bigcap_{\tau \in U}\mathbb{T}_{\tau }$.

\item $\bigvee U$ is the torsion theory associated to the torsion free class 
$\bigcap_{\tau \in U}\mathbb{F}_{\tau }$.
\end{enumerate}

By convention we set $\wedge \emptyset = \mathbf{1}$ and $\vee \emptyset = 
\mathbf{0}$
\end{defn}

\begin{theorem}[\protect\cite{Golan}*{29.1}]
$R-\mathrm{Tors}$ is a frame, \emph{i.e.} $(R-\mathrm{Tors},\wedge ,\vee ,%
\mathbf{1},\mathbf{0})$ is a complete distributive lattice such that for any 
$U\subseteq R-\mathrm{Tors}$ and $\tau \in R-\mathrm{Tors}$ one has 
\begin{equation*}
\tau \wedge \left( \bigvee U\right) =\bigvee_{\sigma \in U}(\tau \wedge
\sigma ).
\end{equation*}
\end{theorem}

In particular $R-\mathrm{Tors}$ is a distributive lattice. The operations $%
\bigwedge $ and $\bigvee $ allow to attach two torsion theories to a module $%
M$: the least torsion theory $\xi (M)$ with respect to which $M$ is torsion
and the greatest torsion theory $\chi (M)$ with respect to which $M$ is
torsionfree and 
\begin{equation*}
\xi (M)=\bigwedge \{\tau \in {R}\text{{-}}{\mathrm{Tors}}\mid M\in \mathbb{T}%
_{\tau }\}.\qquad \chi (M)=\bigvee \{\tau \in R\text{-}\mathrm{Tors}\mid
M\in \mathbb{F}_{\tau }\}.
\end{equation*}

\begin{lem}
\label{rel_xi_chi} Let $M$ be a left $R$-module and $\tau \in R$-$\mathrm{%
Tors}$. Then $\xi (M)\not\leq \tau $ if and only if $\tau \leq \chi (M/N)$
for some proper $N\subset M$.
\end{lem}

\begin{proof}
If $\xi (M)\not\leq \tau $, then $M\not\in \mathbb{T}_{\tau }$. Hence $%
N=\tau (M)\subsetneq M$ with $M/N\in \mathbb{F}_{\tau }$, \emph{i.e.} $\tau
\leq \chi (M/N)$. On the other hand, if $\tau \leq \chi (M/N)$, then $%
M/N\not\in \mathbb{T}_{\tau }$ and hence $\xi (M)\not\leq \tau $.
\end{proof}

\begin{theorem}
Let $M$ be a left $R$-module. $\xi(M)$ is hollow in $R-\mathrm{Tors}$ if and
only if $\{ \chi(M/N) \mid N\subset M\}$ is directed.
\end{theorem}

\begin{proof}
Suppose that $\xi (M)$ is irreducible in $R-\mathrm{Tors}^{\circ }$ and let $%
N,L\subset M$. Since $\xi (M)\not\leq \chi (M/N),\chi (M/L)$ and $\xi (M)$
irreducible, also $\xi (M)\not\leq \tau =\chi (M/L)\vee \chi (M/N)$. By
Lemma \ref{rel_xi_chi} there exists $K\subset M$ such that $\chi (M/L)\vee
\chi (M/N)\leq \chi (M/K)$, \emph{i.e.}, the indicated set is directed.

On the other hand, assume that $\{\chi (M/N)\mid N\subset M\}$ is directed.
If $\xi (M)\not\leq \tau ,\sigma $ for some $\tau ,\sigma \in R-\mathrm{Tors}
$, then by Lemma \ref{rel_xi_chi} there exist $N,L\subset M$ with $\tau \leq
\chi (M/N)$ and $\sigma \leq \chi (M/L)$. By hypothesis there exists $%
K\subset M$ with 
\begin{equation*}
\tau \vee \sigma \leq \chi (M/N)\vee \chi (M/L)\leq \chi (M/K).
\end{equation*}%
By Lemma \ref{rel_xi_chi}, $\xi (M)\not\leq \chi (M(K)$ and therefore $\xi
(M)\not\leq \tau \vee \sigma $.
\end{proof}

We will characterize completely irreducible elements in $R-\mathrm{Tors}%
^{\circ }$. Note that any torsion theory $\tau \in R-\mathrm{Tors}$ is the
join of torsion theories of the form $\xi (R/I)$ with $R/I$ being $\tau $%
-torsion, \emph{i.e.} 
\begin{equation*}
\tau =\bigvee \{\xi (R/I)\mid R/I\mbox{ is $\tau$-torsion }\}.
\end{equation*}%
Hence, any torsion theory $\tau $ that is completely irreducible in $R-%
\mathrm{Tors}^{\circ }$ is of the form $\tau =\xi (R/I)$ for some left ideal 
$I$ of $R$. A module $M$ such that $M$ is either $\tau $-torsion or $\tau $%
-torsionfree for any $\tau \in R-\mathrm{Tors}$ is called \emph{decisive}.

\begin{theorem}
An element $\tau \in R-\mathrm{Tors}$ is completely hollow in $R-\mathrm{Tors%
}$ if and only if $\tau=\xi(M)$ for a (cyclic) decisive module $M$.
\end{theorem}

\begin{proof}
Suppose $\tau $ is completely irreducible in $R-\mathrm{Tors}^{\circ }$. Let 
\begin{equation*}
\tau ^{\Delta }=\bigvee \{\sigma \in R-\mathrm{Tors}\mid \tau \not\leq
\sigma \}.
\end{equation*}%
Then $\tau \not\leq \tau ^{\Delta }$ as $\tau $ is completely irreducible in 
$R-\mathrm{Tors}^{\circ }$. In particular $\mathbb{F}_{\tau }\not\supseteq 
\mathbb{F}_{\tau ^{\Delta }}$. Hence there exists a non-zero cyclic $\tau $%
-torsion module $M$ that is $\tau ^{\Delta }$-torsionfree. Hence $\xi
(M)\leq \tau $ and if $\tau \not\leq \xi (M)$ one had $\xi (M)\leq \tau
^{\Delta }$, \emph{i.e.} $M\in \mathbb{T}_{\tau ^{\Delta }}\cap \mathbb{F}%
_{\tau ^{\Delta }}=\{0\}$. Thus $\tau \leq \xi (M)$, \emph{i.e.} $\tau =\xi
(M)$. Moreover for any $\sigma \in R-\mathrm{Tors}$: if $\tau \leq \sigma $,
then $M$ is $\sigma $-torsion. Otherwise $\tau \not\leq \sigma $ and $\sigma
\leq \tau ^{\Delta }$, which shows that $M$ is $\sigma $-torsionfree, as $M$
is $\tau ^{\Delta }$-torsionfree. Hence, $M$ is decisive.

On the other hand, if $\tau $ is of the given form and $U$ is any subset of $%
R-\mathrm{Tors}$, such that $\tau \not\leq \sigma $ for all $\sigma \in U$,
then $M\in \mathbb{F}_{\sigma }$ for all $\sigma \in U$ and $M\in
\bigcap_{\sigma \in U}\mathbb{F}_{\sigma }=\mathbb{F}_{\bigvee U}$, \emph{%
i.e.} $\tau =\xi (M)\not\leq \bigvee U$.
\end{proof}

Examples of decisive modules are strongly prime modules in the sense of \cite%
{HandelmanLawrence} which also define Rosenberg's left spectrum (see \cite%
{Rosenberg}).



\end{document}